\documentclass{amsart}
\usepackage{amsmath,amssymb,amsthm,color,url}

\newtheorem{thr}{Theorem}
\newtheorem{lem}[thr]{Lemma}
\newtheorem{cor}[thr]{Corollary}

\newtheorem{obs}[thr]{Observation}

\theoremstyle{definition}
\newtheorem{defn}[thr]{Definition}

\newtheorem{ex}[thr]{Example}

\theoremstyle{remark}

\def\B{\mathcal{B}}
\def\U{\mathcal{U}}
\def\W{\mathcal{W}}
\def\F{\mathcal{F}}
\def\H{\mathcal{H}}

\def\Rc{\mathcal{R}}
\def\S{\mathcal{S}}

\begin{document}

\title[How hard is the tensor rank?]{How hard is the tensor rank?}

\author{Yaroslav Shitov}
\address{National Research University Higher School of Economics, 20 Myasnitskaya Ulitsa, Moscow, Russia 101000}
\email{yaroslav-shitov@yandex.ru}


\begin{abstract}
We investigate the computational complexity of \textit{tensor rank}, a concept that plays fundamental role in different topics of modern applied mathematics. For tensors over any integral domain $\Rc$, we prove that the rank problem is polynomial time equivalent to solving a system of polynomial equations over $\Rc$. Our result gives a complete description of the algorithmic complexity of tensor rank and allows one to solve several known open problems. In particular, the tensor rank over $\mathbb{Z}$ turns out to be undecidable, which answers the question posed by Gonzalez and Ja'Ja' in 1980. We generalize our result and prove that the \textit{symmetric rank} admits a similar description of computational complexity as the one we give for usual rank. In particular, computing the symmetric rank of a rational tensor is shown to be NP-hard, which proves a recent conjecture of Hillar and Lim. As a byproduct of our approach, we get a similar characterization of the algorithmic complexity of the \textit{minimal rank matrix completion} problem, which gives a complete answer to the question discussed in 1999 by Buss, Frandsen, and Shallit.
\end{abstract}

\maketitle

\section{Introduction}

The \textit{rank} of a tensor $T$ is the smallest integer $r$ such that $T$ can be decomposed into a sum of $r$ simple tensors. This concept has been introduced eighty years ago~\cite{Hitch} and became a fundamental tool in different branches of modern science. We note that the rank decompositions have proved to be useful for studies in statistics~\cite{SH}, signal processing~\cite{LC}, complexity of computation~\cite{Str}, psychometrics~\cite{CC}, linguistics~\cite{Smo}, chemometrics~\cite{CLD}, and we refer the reader to the monograph~\cite{Lan} for a more complete survey of applications. Various recent publications, including~\cite{CGLM, KB, OO, SBG}, explain the relevance of symmetric decompositions for problems in pure mathematics and engineering. The importance for applications stimulates the research on algorithms solving the rank decomposition problems and their computational complexity.

The first step towards understanding the computational complexity has been made by H\r{a}stad~\cite{Hast}, who proved that the tensor rank of a rational matrix is NP-hard to compute. Another result of H\r{a}stad states that the rank decomposition problem is NP-complete in the case of finite fields. A recent paper of Hillar and Lim~\cite{HL} shows that H\r{a}stad's approach works well enough to prove the NP-hardness of the rank decomposition problem over $\mathbb{R}$ and $\mathbb{C}$. However, there has been no complete characterisation of the complexity of tensor rank, and many questions on it remained open. For instance, it remained unknown whether the tensor rank over $\mathbb{Z}$ can be computed by a finite algorithm, and there were no known lower bounds on the complexity of rational tensor rank better than the NP-hardness.

Our paper aims to fill this gap and give a complete description of the algorithmic complexity of tensor rank. As we show in our paper, the problem of computing the tensor rank with respect to any integral domain is polynomial time equivalent to the problem of deciding if a given system of polynomial equations has a solution over this integral domain. We prove the same result on the computational complexity of \textit{symmetric rank} for symmetric tensors over a field. As a byproduct of our approach, we get a similar characterization of the algorithmic complexity of the \textit{minimal rank matrix completion} problem. In the following section, we provide the precise formulation of our result and interesting corollaries of it.

\section{The formulation of main result and corollaries}

We deal with order-three tensors over a commutative ring $\Rc$. In other words, a \textit{tensor} is a three-dimensional array $T$ with elements $T(i|j|k)$ in $\Rc$, where $i,j,k$ run over corresponding indexing sets $I,J,K$. We write $T\in\Rc^{I\times J\times K}$ and say that $|I|\times|J|\times|K|$ is the \textit{size} of $T$. A tensor $T$ is called \textit{symmetric} if $I=J=K$ and we have $T(i_1|i_2|i_3)=T(j_1|j_2|j_3)$ whenever $(i_1,i_2,i_3)$ is a permutation of $(j_1,j_2,j_3)$.

Given three vectors $a\in\Rc^I$, $b\in\Rc^J$, $c\in\Rc^K$, we define the tensor $a\otimes b\otimes c\in\Rc^{I\times J\times K}$ by setting its $(i,j,k)$th entry to be $a_ib_jc_k$. Tensors arising in this way are called \textit{rank one} or \textit{simple} with respect to $\Rc$. It is important to note that, if we allow the vectors $a,b,c$ to contain elements not from $\Rc$ but rather from some extension $\S$, we may possibly get a different set of simple tensors. Actually, it is well known (see~\cite{dSLi}) that the rank of a tensor may depend on an extension of $\Rc$ in which we take the entries of rank-one tensors in sum decompositions. So for any extension $\S\supset\Rc$, we denote by $\operatorname{rank}_{\S} T$ the \textit{rank} of $T$ \textit{with respect to} $\S$, that is, the smallest number of tensors which are simple with respect to $\S$ and sum to $T$. The following example shows that the rank with respect to a ring may depend on the extension even for matrices (of which we think of as $m\times n\times 1$ tensors).



\begin{ex}\label{exdep}(See Example~17 in~\cite{myfpm}.)
\textit{The rank of the matrix} $$\begin{pmatrix}
x&-z&0\\
0&y&x\\
y&0&z
\end{pmatrix}$$
\textit{is three over the ring $\mathbb{R}[x,y,z]$ and two over the field} $\mathbb{R}(x,y,z)$.
\end{ex}

In order to discuss the computational complexity, we need to impose some additional requirements on $\Rc$. In particular, we assume that the elements of $\Rc$ can be encoded by strings in some finite alphabet so that the addition and multiplication in $\Rc$ can be performed by polynomial time algorithms. We do not need such an assumption for an extension $\S$ in which we take the coefficients of sum decompositions, and this $\S$ can be arbitrary. In particular, our considerations are valid for usual real ranks of rational tensors, and this case corresponds to $\Rc=\mathbb{Q}$, $\S=\mathbb{R}$. We are ready to formulate one of the main results.


%
%
%

\begin{thr}\label{thrm2}
Let $\Rc\subseteq \S$ be integral domains, and let $f_1,\ldots,f_p$ be polynomials with coefficients in $\Rc$. There is a polynomial time algorithm that constructs an order-three tensor $\mathcal{T}$ over $\Rc$ and an integer $r$ such that the following are equivalent:

\noindent (1) the equations $f_1=0,\ldots,f_p=0$ have a simultaneous solution in $\S$;

\noindent (2) the rank of $\mathcal{T}$ with respect to $\S$ does not exceed $r$. 

Moreover, these $\mathcal{T}$ and $r$ do not depend on the choice of $\S$.
\end{thr}

Let us explain why Theorem~\ref{thrm2} can be seen as a complete description of the algorithmic complexity of tensor rank. This theorem presents a polynomial time algorithm that, for any finite family $F$ of polynomials over $\Rc$, constructs a tensor whose rank with respect to $\S$ is at most $r$ if and only if the zero locus of $F$ contains a point with coordinates in $\S$. On the other hand, the problem of decomposing a tensor with entries in $\Rc$ is naturally formulated as a system of polynomial equations with coefficients in $\Rc$, so Theorem~\ref{thrm2} proves that these two problems are in fact equivalent.

\begin{thr}\label{cor80}
Let $\Rc\subseteq \S$ be integral domains. The problem of deciding if a tensor with entries in $\Rc$ has rank at most $r$ with respect to $\S$ is polynomial time equivalent to the problem of deciding if a given system of polynomial equations with coefficients in $\Rc$ has a solution in $\S$.
\end{thr}

Let us discuss several corollaries of Theorem~\ref{thrm2} that are of independent interest. In particular, we see that the rational tensor rank is polynomial time equivalent to deciding if a given \textit{Diophantine equation}\footnote{That is, a polynomial equation with integral coefficients.} has a rational solution\footnote{This is a solution to the problem of Bl\"{a}ser, see Open Problem~2 on page~119 in~\cite{Blas}.}. The latter problem is a variant of \textit{Hilbert's tenth problem}, and it is widely believed to be undecidable although this claim remains open despite the extensive research~\cite{KR, Koenin, MR, Maz, Poon}. Therefore, Theorem~\ref{thrm2} can be seen as a conditional proof of undecidability of rational tensor rank, which would confirm Conjecture~13.3 in the paper~\cite{HL} by Hillar and Lim. We note that the actual \textit{Hilbert's tenth problem}, which asks if a given Diophantine equation has an \textit{integral} solution, was famously proved to be undecidable in the $70$'s, see~\cite{MR}. This gifts to us the following corollary of Theorem~\ref{thrm2}.

\begin{cor}\label{corz}
Tensor rank over $\mathbb{Z}$ is undecidable.
\end{cor}

Corollary~\ref{corz} answers the question by Gonzalez and Ja'Ja' dating back to 1980, see page~77 of~\cite{GJ}. The authors did not explicitly mention the concept of tensor rank, but they formulated an equivalent question on the multiplicative complexity of simultaneous computing of bilinear forms. It is well known that these formulations are equivalent (see~\cite{Str}), and we will further use this equivalence and comment it on in Section~3.

Another interesting corollary concerns the problem of computing the real rank of a rational tensor. By Theorem~\ref{thrm2}, this problem turns out to be equivalent to the so-called \textit{existential theory of the reals}, see~\cite{Mato}. We can get a similar result for the rank of a rational tensor computed with respect to any field of characteristic zero. As we explain in Section~3, the existence of a solution of a system of Diophantine equations is an NP-hard problem over any integral domain, and this gives us the following corollary of Theorem~\ref{cor80}.

\begin{cor}\label{cornph}
Tensor rank is NP-hard over any integral domain.
\end{cor}

This corollary is a generalization of the results by H\r{a}stad~\cite{Hast} and Hillar and Lim~\cite{HL}. H\r{a}stad proved the NP-hardness of tensor rank over the rationals and finite fields, and Hillar and Lim used his approach to show the NP-hardness for the reals and complex numbers.

\medskip

Now let $T$ be a symmetric tensor. Recall that the \textit{symmetric rank} of $T$ \textit{with respect to} a field $\S$ is the smallest number of simple symmetric tensors over $\S$ whose linear span contains $T$. We denote this quantity by $\operatorname{srank}_{\S} T$. As we show in our paper, the analogue of Theorem~\ref{thrm2} is valid for \textit{symmetric ranks} of tensors over a field. This result has the following important consequence, which was conjectured by Hillar and Lim.

\begin{thr}\label{prob1}\emph{(Conjecture 13.2 in~\cite{HL}.)}
Let $S\in\mathbb{Q}^{n\times n\times n}$ be a symmetric tensor. Computing the symmetric rank of $S$ with respect to any field $\mathbb{K}\supset\mathbb{Q}$ is NP-hard.
\end{thr}

In view of the known hardness results for tensor ranks, many authors believed that the symmetric rank is hard to compute as well. In particular, Ottaviani and Oeding in~\cite{OO} and Robeva in~\cite{Robeva} explicitly write that the symmetric rank is hard to compute, although they are refer to the paper~\cite{HL} in which the NP-hardness is just a conjecture. The recent paper~\cite{GKPT} also mentions this conjecture and solves a related problem on the partial derivatives of polynomials. Several papers, including~\cite{BGI, BCMT}, were devoted specifically to the problem of computing symmetric rank, but the result of Theorem~\ref{prob1} remained unknown until now.

\medskip

One of the main tools used in our paper is the \textit{minimal rank matrix completion} problem. We say that a matrix $M$ with entries in $\Rc\cup\{*\}$ is an \textit{incomplete matrix} over $\Rc$, and any matrix obtained by replacing the $*$'s with elements in $\S$ is called a \textit{completion} of $M$ over $\S$. What is the smallest value that the rank of a completion of a given incomplete matrix may take? As a byproduct of our approach, we will get the description of the complexity of this problem as in Theorem~\ref{cor80}.

\begin{thr}\label{cor79}
Let $\Rc\subseteq \S$ be commutative rings. The problem of deciding if a given incomplete matrix with entries in $\Rc\cup\{*\}$ has a completion of rank three with respect to $\S$ is polynomial time equivalent to the problem of deciding if a given system of polynomial equations with coefficients in $\Rc$ has a solution over $\S$.
\end{thr}

The author believes that this result is new even for the case of fields. The NP-hardness of the problem being discussed was explicitly mentioned by Derksen (\cite{Derk}) in the case of fields, but I do not see how can the source he refers to be useful to prove it. 
Numerous related matrix completion problems are shown to be NP-hard, and there are those for which the complexity is described completely. In particular, a result similar to Theorem~\ref{cor79} holds for the version of minimal rank problem in which some of the $*$ entries may be required to take the same value, see~\cite{BFS}.
In fact, the complexity of our version was discussed in~\cite{BFS} as well, but the authors were not able to give any non-trivial lower bound on it. Also, Laurent writes in~\cite{Laur} that the minimal rank completion seems to be a difficult task, but again she does not mention any particular result on the complexity of this problem. As said above, our Corollary~\ref{cor79} does not only prove the NP-hardness of the minimal rank completion over fields, but fully determines the complexity of this problem over any commutative ring. Other related completion problems whose complexity has been known are the \textit{Euclidean distance completion} (\cite{Laur}), \textit{minimal rank sign pattern completion} (\cite{BK}), and the problem of minimizing the rank of matrices \textit{fitting} a given graph (\cite{Peet}). We note that the approximate version of the minimal rank completion problem plays an important role in applied mathematics (\cite{CP, onemore}), and the version in which the $*$ entries are taken at random is particularly important (\cite{CR}).

\medskip

The rest of the paper is devoted to the proofs of the results mentioned in this section. In Section~3, we recall several basic results on tensor decompositions and computational complexity, and we explain how all the results of the present section follow from Theorems~\ref{thrm2}, \ref{prob1}, and~\ref{cor79}. In Section~4, we prove Theorem~\ref{cor79}, and our approach can be seen as a variation of the recent investigation~\cite{mypsd} of the complexity of the psd version of the matrix completion problem. In Section~5, we employ the approach recently used by Derksen~\cite{Derk} and reduce the matrix completion problem to tensor rank. A direct application of Derksen's approach would be sufficient to prove Theorem~\ref{thrm2} in the case when $\Rc$ is a field, but this would not be enough to get Corollaries~\ref{corz} and~\ref{cornph}. We managed to adapt Derksen's construction to work over any integral domain, which allowed us to reduce the matrix completion problem to tensor rank and complete the proof of the main result. In Section~6, we prove the analogue of Theorem~\ref{thrm2} for symmetric ranks and deduce Theorem~\ref{prob1} from this result.

\section{Preliminaries}

All our results on the algorithmic complexity are based on the \textit{Turing model} of computation. As said above, we always assume that the elements of a ring $\Rc$ can be stored as finite strings, and there is a specific Turing machine that halts in polynomial time and, given two elements of $\Rc$, returns their sum and product. In particular, $\Rc$ can be taken to be a finite ring, the integers $\mathbb{Z}$, the rationals $\mathbb{Q}$, but not the reals $\mathbb{R}$ or complex numbers $\mathbb{C}$. However, we can still discuss the ranks of tensors with respect to $\mathbb{R}$ and $\mathbb{C}$ because we do not impose any restriction to an extension $\S$ with respect to which these ranks are computed. The polynomials over $\Rc$ are stored as sums of monomials, and every monomial is represented as a scalar in $\Rc$ multiplied by a product of single variables. When we say that two decision problems $\Pi_1$ and $\Pi_2$ are \textit{polynomial time equivalent}, we mean that there are polynomial time many-one reductions from $\Pi_1$ to $\Pi_2$ and from $\Pi_2$ to $\Pi_1$. In particular, Theorem~\ref{thrm2} presents such a reduction from the problem of deciding if a given system of polynomial equations with coefficients in $\Rc$ has a solution in $\S$ to the problem of deciding if a tensor with entries in $\Rc$ has rank at most $r$ with respect to $\S$. As said in the argument below Theorem~\ref{thrm2}, the converse direction of this reduction is straightforward, so Theorem~\ref{thrm2} implies Theorem~\ref{cor80}.

Since the existence of integral solutions is undecidable for Diophantine equations, Corollary~\ref{corz} indeed follows from Theorem~\ref{cor80}. We deduce Corollary~\ref{cornph} from Theorem~\ref{cor80} with the standard observation whose proof is given for completeness.

\begin{obs}
(See~\cite{KR}.)
Let $\Rc$ be an integral domain. It is NP-hard to tell if a given system of polynomial equations with integral coefficients has a solution in $\Rc$.
\end{obs}

\begin{proof}
We construct a reduction from 3-SAT (see~\cite{Karp}), which is the problem of determining the satisfiability of a formula in conjunctive normal form where each clause is limited to at most three literals. In other words, this problem asks if a given set of conditions of the forms $x_i\vee x_j\vee x_k$ and $x_u\neq x_v$ can be satisfied simultaneously by some assignment of Boolean variables $(x_i)$.

Let $\varphi$ be a formula that satisfies the above restrictions and involves variables $x_1,\ldots,x_n$. We are going to construct a family $F$ of polynomial equations with integral coefficients which have a simultaneous solution in $\Rc$ if and only if $\varphi$ is satisfiable. For any variable $x_i$, we add to $F$ the equation $y_i^2-y_i=0$. Now the solutions of $F$ have to consist of zeros and ones, and we can identify the domains of $x_i$ and $y_i$. Every clause of the form $x_i\vee x_j\vee x_k$ is now equivalent to the equation
$$y_i+y_j+y_k-y_iy_j-y_iy_k-y_jy_k+y_iy_jy_k=1,$$
and the clause $x_u\neq x_v$ is represented by $y_u+y_v=1$.
\end{proof}




Therefore, we have explained how Theorem~\ref{cor80} and Corollaries~\ref{corz},~\ref{cornph} follow from Theorem~\ref{thrm2}. Now we are going to comment on the relation between the tensor rank and multiplicative complexity of \textit{bilinear programs}. We refer the reader to~\cite{GJ, Str} for a definition of a bilinear program, and we recall that the rank of the corresponding tensor equals the minimal number of multiplications such a program needs to perform in order to compute a collection of bilinear forms (see~\cite{GJ, Hast, Str}). This correspondence allows us to use the tensor analogues of results that were initially formulated and proved in terms of multiplicative complexity. In particular, we see that Corollary~\ref{corz} does actually answer the above mentioned question by Gonzalez and Ja'Ja'.

Let $T$ be a tensor in $\Rc^{I\times J\times K}$. We define the $k$th \textit{$3$-slice} of $T$ as a matrix in $\Rc^{I\times J}$ whose $(i,j)$ entry equals $T(i|j|k)$. For all $i\in I$, $j\in J$, we can define the $i$th $1$-slice and $j$th $2$-slice of $T$ in a similar way. Our further considerations will employ the following lemma, which  is a reformulation of Lemma~2 in~\cite{HK}, see also~\cite{Hast}.

\begin{lem}\label{lemcompl}
Let $\F$ be a field, and let $T$ be a tensor in $\F^{I\times J\times K}$ with $K=\{1,\ldots,k\}\cup\{1',\ldots,\tau'\}$. Denote by $S_i$ the $i$th $3$-slice of $T$ and assume that $S_{1'},\ldots,S_{\tau'}$ are linearly independent and rank-one. Then $\operatorname{rank}_\F T$ is equal to $$\tau+\min\operatorname{rank}_\F \mathcal{T}(V_1,\ldots,V_k),$$ where the tensor $\mathcal{T}(V_1,\ldots,V_k)$ is formed by the slices $S_1-V_1,\ldots,S_k-V_k$, and the $V_j$'s belong to the $\F$-linear span of $S_{1'},\ldots,S_{\tau'}$.
\end{lem}





In the remainder of the paper, $\Rc$ and $\S$ denote commmutative rings satisfying $\Rc\subset\S$. Sometimes (and then it is pointed out explicitly) we impose further restrictions on $\Rc,\S$. By $F=\{f_1,\ldots,f_t\}$ we denote a finite subset of $\Rc[x_1,\ldots,x_n]$.

\section{The complexity of minimal rank matrix completion}

In this section, we determine the algorithmic complexity of the minimal rank matrix completion problem. Namely, we are going to prove the result as in Theorem~\ref{thrm2} but for the minimal rank completion instead of tensor rank.

We consider a finite set $F\subset\Rc[x_1,\ldots,x_n]$ of polynomials, and we are going to construct an incomplete matrix whose minimal rank is three if and only if the equations corresponding to $F$ have simultaneous solutions. For any monomial $p=\xi x_{i_1}\ldots x_{i_k}$ (with $\xi$ in $\Rc$), we define $$\sigma(p)=\{\pm1,\pm\xi,\pm x_{i_1},\pm x_{i_1}x_{i_2},\ldots, \pm x_{i_1}\ldots x_{i_k}, \pm p\}.$$ For a general polynomial $f=p_1+\ldots+p_s$, we define $$\sigma(f)=\sigma(p_1)\cup\ldots\cup\sigma(p_s)\cup\{0,\pm p_1,\pm(p_1+p_2),\ldots,\pm f\}\cup\{\pm x_1,\ldots,\pm x_n\},$$
and for a set of polynomials $F=\{f_1,\ldots,f_t\}$, we define $\sigma(F)=\sigma(f_1)\cup\ldots\cup\sigma(f_t)$.

Clearly, the construction of the set $\sigma=\sigma(F)$ can be done in time polynomial in the length of $F$, and we denote by $\sigma^3$ the set of all triples of elements in $\sigma$. We denote by $\H=\H(F)$ the set of those vectors in $\sigma^3$ that have one of the coordinates equal to $1$ or $-1$. We denote by $\U=\U(x_1,\ldots,x_n)$ the matrix whose columns are vectors in $\H$, and we define $\W(x_1,\ldots,x_n)=\U^\top \U$. Now we are ready to present the key example of an incomplete matrix $\B$ with rows and columns indexed by vectors in $\H$.

\begin{defn}\label{defnb}
For all $u,v\in\H$, we define the polynomial $\delta(u,v)\in\Rc[x_1,\ldots,x_n]$ as the dot product of $u$ and $v$, or, equivalently, by the formula $\delta(u,v)=\W(u|v)$. We define the matrix $\B$ with entries in $\Rc\cup\{*\}$ as follows:

\noindent (1) If $\delta(u,v)$ is a constant polynomial, then we define $\B(u|v)=\delta(u,v)\in\Rc$;

\noindent (2) If $\delta(u,v)$ equals one of the polynomials in $F$, then we set $\B(u|v)=0$;

\noindent (3) In the remaining cases (that is, when $\delta(u,v)$ is neither a constant nor a polynomial in $F$), we set $\B(u|v)=*$.
\end{defn}

Now let $\S$ be an extension of $\Rc$ and assume that the polynomial equations $f_1=0,\ldots,f_t=0$ have a simultaneous solution $(\xi_1,\ldots,\xi_n)$ over $\S$. Clearly, the matrix $\W(\xi_1,\ldots,\xi_n)$ is a rank-three completion of $\B$, and we are going to show that all rank-three completions arise in this way up to the natural action of the group of invertible $3\times3$ matrices.

\begin{lem}\label{lem73}
Let $P,L$ be $3\times|\H|$ matrices over $\S$ such that the matrix $P^\top L$ is a completion of $\B$. Let $C$ be the matrix obtained by taking the columns of $L$ with indices in $E=\{(1,0,0), (0,1,0), (0,0,1)\}$. Then we have $C^\top P=C^{-1} L=\U(\xi_1,\ldots,\xi_n)$, where $(\xi_1,\ldots,\xi_n)$ is a simultaneous solution of the equations $f_1=0,\ldots,f_t=0$.
\end{lem}

\begin{proof}
\textit{Step 1.} Since the $3\times3$ submatrix of $\B$ with row and column indices in $E$ is the unity matrix, we get that the submatrix of $C^\top P$ formed by the rows with indices in $E$ is an inverse of $C$. Since the transformation $(P,L)\to (C^\top P,C^{-1} L)$ does not change the property of $P^\top L$ to be a completion of $\B$, we can assume without loss of generality that $C$ is the unity matrix.

\textit{Step 2.} Denoting by $p_u, l_u$ the $u$th columns of $P,L$, we get \begin{equation}\label{eqpr1}p_u\cdot l_v=\B(u|v) \mbox{ $ $ whenever $ $}\B(u|v)\neq*.\end{equation} By the result of Step~1, we have $p_{(1,0,0)}=(1,0,0)$, $p_{(0,1,0)}=(0,1,0)$, $p_{(0,0,1)}=(0,0,1)$, $l_{(1,0,0)}=(1,0,0)$, $l_{(0,1,0)}=(0,1,0)$, $l_{(0,0,1)}=(0,0,1)$. Combining these equations with~\eqref{eqpr1}, we get that $P(j|u)=\U(j|u)$ whenever $\U(j|u)$ is a constant.

\textit{Step 3.} Using Step~2, we get that, for any variable $x_i$, there is $y_i\in\S$ such that $l_{(1,0,x_i)}=(1,0,y_i)$. In what follows, $y_i$ denotes the third coordinate of the vector $l_{(1,0,x_i)}$, and we write $x=(x_1,\ldots,x_n)$, $y=(y_1,\ldots,y_n)$.

\textit{Step~4.} We are going to complete the proof by checking that $f(y)=0$, for all $f\in F$. We say that the label $u=(a(x),b(x),c(x))$ is \textit{$p$-good} (or \textit{$l$-good}) if the vector $p_u$ (or $l_u$, respectively) equals $(a(y),b(y),c(y))$. By Step~2, the labels consisting of constants are necessarily $p$-good and $l$-good.

\textit{Step~5.} Now assume that a vector $(g,0,h)$ is $l$-good. The vector $p_{(-h,g,g)}$ is then orthogonal to $l$-good vectors $(g,0,h)$ and $(0,-1,1)$, so we see that $p_{(-h,g,g)}$ is collinear to $(-h,g,g)$. Since every column of $\U$ has at least one constant entry, the result of Step~2 implies that $p_{(-h,g,g)}$ and $(-h,g,g)$ are in fact equal. In other words, $(-h,g,g)$ is a $p$-good vector, and we note that the vector $(g,h,0)$ is $l$-good for a similar reason, --- namely, because $l_{(g,h,0)}$ is orthogonal to $p$-good vectors $(0,0,1)$ and $(-h,g,g)$.

\textit{Step~6.} If $(g,0,h)$ is $l$-good, the vector $(-h,0,g)$ is $p$-good because it is orthogonal to $l$-good vectors $(g,0,h)$ and $(0,1,0)$. The symmetry and Step~5 imply that, in the case when $(g,0,h)$ is $l$-good, any permutation of $(g,h,0)$ is $l$-good and any permutation of $(-h,0,g)$ is $p$-good.

\textit{Step~7.} Let us now assume that $(1,0,\alpha)$, $(1,0,\beta)$ are $l$-good.

\textit{Step~7.1.} Assume $\alpha+\beta\in\sigma$. We see that $(-1,1,\alpha)$ is $l$-good because it is orthogonal to the vectors $(1,1,0)$ and $(0,-\alpha,1)$, which are $p$-good by Steps~4 and~6. Now we see that $(-\beta,-\alpha-\beta,1)$ is $p$-good because it is orthogonal to the $l$-good vectors $(-1,1,\alpha)$ and $(1,0,\beta)$. Finally, the vector $(0,1,\alpha+\beta)$ is $l$-good because it is orthogonal to the $p$-good vectors $(-\beta,-\alpha-\beta,1)$ and $(1,0,0)$. The vector $(1,0,\alpha+\beta)$ is $l$-good by Step~6.

\textit{Step~7.2.} Assume $\alpha\beta\in\sigma$. The vector $(\alpha\beta,1,\alpha)$ is $l$-good because it is orthogonal to the vectors $(0,-\alpha,1)$ and $(1,0,-\beta)$, which are $p$-good by Step~6. The vector $(1,-\alpha\beta,0)$ is $p$-good because it is orthogonal to the $l$-good vectors $(0,0,1)$ and $(\alpha\beta,1,\alpha)$. Finally, we see that $(\alpha\beta,1,0)$ is a $l$-good vector because it is orthogonal to $p$-good vectors $(1,-\alpha\beta,0)$ and $(0,0,1)$. Step~5 implies that $(1,0,\alpha\beta)$ is $l$-good as well.

\textit{Step~8.} The results of Step~7 show that the vector $(1,0,s)$ is $l$-good for all $s\in\sigma$. For all $f\in F$, we get $p_{(0,0,1)}\cdot l_{(1,0,f)}=0$ because $\B((0,0,1)|(1,0,f))=0$, which means that $(0,0,1)\cdot(1,0,f(y))=0$ or $f(y)=0$. The proof is complete. 
\end{proof}

The following corollary is immediate from Lemma~\ref{lem73}.

\begin{cor}\label{cor74}
The matrix $\B(F)$ admits a completion of rank three with respect to $\S$ if and only if the equations $f_1=0,\ldots,f_t=0$ have a simultaneous solution over $\S$.
\end{cor}

\begin{proof}
As said above, the matrix $\W(\xi_1,\ldots,\xi_n)$ is a rank-three completion of $\B$, provided that $(\xi_1,\ldots,\xi_n)$ is a simultaneous solution of the polynomial equations $f_1=0,\ldots,f_t=0$. Conversely, if there is no such a solution over $\S$, then by Lemma~\ref{lem73} the matrix $\B$ admits no completion of rank three with respect to $\S$.
\end{proof}

Corollary~\ref{cor74} proves an analogue of Theorem~\ref{thrm2} for the minimal rank matrix completion problem. In particular, we get Theorem~\ref{cor79} as well as the matrix completion analogues of the complexity and decidability results formulated in the previous section for tensor rank. More than that, our result for the minimal rank completion allows the rings $\Rc$ and $\S$ to contain zero divisors, while in the subsequent consideration of tensor ranks we will often assume that $\S$ is an integral domain.

\section{Proof of Theorem~\ref{thrm2}}

In this section, we assume that $F$ is the family of polynomials and $\B$ is the incomplete matrix as in the previous section. We begin with a technical lemma that we will need in future considerations. 

\begin{lem}\label{lem75}
Let $\S$ be an integral domain, and let $\F$ be a field containing $\S$. Assume that $W_1,W_2,W_3$ are rank-one matrices over $\F$ such that $W_1+W_2+W_3$ is a completion of $\B$. Assume that, for some $\lambda_1,\lambda_2,\lambda_3\in\S$, some rank-one matrix $W_0$ coincides with $\lambda_1 W_1+\lambda_2 W_2+\lambda_3 W_3$ everywhere except those entries that are $*$'s in $\B$. Then there is a $\mu\in\F$ such that $W_0$ is one of $\mu W_1,\mu W_2,\mu W_3$.
\end{lem}

\begin{proof}
We define the $E$\textit{-submatrix} of $\B$ as the one formed by the rows and columns with indices in $E=\{(1,0,0), (0,1,0), (0,0,1)\}$. We note that this submatrix is the unity matrix and, in particular, it does not contain $*$'s. The corresponding $E$-submatrices of $W_1,W_2,W_3$ are rank-one and sum to a rank-three matrix, so the rank of the $E$-submatrix of $\lambda_1 W_1+\lambda_2 W_2+\lambda_3 W_3$ equals the number of non-zero $\lambda_i$'s. Therefore, it suffices to consider the case when $W_0$ coincides with $W_3$ at the non-$*$ entries. We are going to show that in fact $W_0=W_3$.

We write $W_j=a_j b_j^\top$ with $a_j,b_j\in\F^\H$, and we can arrange these vectors so that the coordinates of $a_0$ (or $b_0$) with indices in $E$ are equal to the corresponding coordinates of $a_3$ (or $b_3$, respectively).

We define $P$ as the matrix formed by the rows $a_1,a_2,a_3$ and $L$ as the matrix formed by the rows $b_1,b_2,b_3$. We have that $P^\top L=W_1+W_2+W_3$ is a completion of $\B$, and Lemma~\ref{lem73} implies that $C^{-1}L=\U(\xi_1,\ldots,\xi_n)$, where $(\xi_1,\ldots,\xi_n)$ is a solution of $f_1=0,\ldots,f_t=0$, and $C$ is the $3\times3$ matrix formed by the columns of $L$ with indices in $E$. Similarly, we define $Q$ as the matrix formed by the rows $b_1,b_2,b_0$, and we get that $C^{-1}Q=\U(\psi_1,\ldots,\psi_n)$, where $(\psi_1,\ldots,\psi_n)$ is a solution of $f_1=0,\ldots,f_t=0$. The entries of $Q-L$ are all zero except possibly those in the third row, so the matrix \begin{equation}\label{eqeqeq}\U(\psi_1,\ldots,\psi_n)-\U(\xi_1,\ldots,\xi_n)=C^{-1}(Q-L)\end{equation} has rank at most one. By its definition, the set $\H$ consists of vectors one of whose coordinates is constant, so that the matrix~\eqref{eqeqeq} has a zero in every column. Since the rank of~\eqref{eqeqeq} is at most one, it has a zero row, which suffices to conclude that $(\psi_1,\ldots,\psi_n)=(\xi_1,\ldots,\xi_n)$. Therefore, the matrix~\eqref{eqeqeq} is zero, which means that $Q=L$ or $b_0=b_3$. By symmetry we have $a_0=a_3$ and $W_0=W_3$.
\end{proof}

Now we are going to construct a reduction from the matrix completion problem to tensor rank. We enumerate by $k_1=(i_1,j_1),\ldots,k_\tau=(i_\tau,j_\tau)$ the entries which are $*$'s in $\B$, so $\tau$ is the number of such entries. We define the tensor $T=T(\B)$ in $\Rc^{\H\times\H\times K}$, where $K=\{0,1,\ldots,\tau\}$, as follows:

\noindent (1) $T(u|v|t)=\B(u|v)$ if $t=0$ and $\B(u|v)\neq*$;

\noindent (2) $T(u|v|t)=1$ if $k_t=(u,v)$;

\noindent (3) $T(u|v|t)=0$ in the remaining cases.

In particular, the zeroth $3$-slice of $T$ is obtained from $\B$ by replacing the $*$'s with zeros, and the other $3$-slices are the $(i_t,j_t)$ matrix units. This construction has been proposed in~\cite{Derk} by Derksen, who showed that $\operatorname{rank}_{\S} T\geqslant k+r$, where $r$ is the minimal rank of any completion of $\B$, in the case when $\S$ is a field. We are going to adapt Derksen's technique and prove the same result for any integral domain $\S$. Namely, we will prove that $\operatorname{rank}_{\S} T=\tau+3$ if and only if $\B$ admits a completion of rank three with respect to $\S$. This result is sufficient to prove Theorem~\ref{thrm2} because, by Corollary~\ref{cor74}, the matrix $\B$ admits a completion of rank three with respect to $\S$ if and and only if the equations corresponding to polynomials in $F$ have a simultaneous solution over $\S$.

\begin{lem}\label{lem76}
If $\S$ is an integral domain, then $\operatorname{rank}_{\S} T\geqslant \tau+3$.
\end{lem}

\begin{proof}
Since $\S$ is an integral domain, there is a field containing $\S$, and the assertion follows from the above mentioned result by Derksen.
\end{proof}

\begin{lem}\label{lem77}
If $\B$ admits a completion of rank three with respect to an integral domain $\S$, then $\operatorname{rank}_{\S} T\leqslant \tau+3$.
\end{lem}

\begin{proof}
If $B$ is such a completion, then we can get a tensor $B_0$ of rank three over $\S$ by setting $B_0(u|v|t)=B(u|v)$ if $t=0$ and $B_0(u|v|t)=0$ otherwise. Further, we define a simple tensor $S_t$ whose entries are all zeros except $S_t(i_t|j_t|0)=-B(i_t|j_t)$ and $S_t(i_t|j_t|t)=1$. We get $T=B_0+S_1+\ldots+S_\tau$, so the result follows.
\end{proof}

\begin{lem}\label{lem78}
If $\S$ is an integral domain and $\operatorname{rank}_{\S} T\leqslant \tau+3$, then $\B$ admits a completion of rank three with respect to $\S$.
\end{lem}

\begin{proof}
Let $\F$ be a field containing $\S$, and let $$T=S_1+\ldots+S_{\tau+3}$$ be a decomposition of $T$ into a sum of tensors that are simple with respect to $\S$. Let $V$ be the $\F$-linear space spanned by the zeroth $3$-slices of $S_1,\ldots,S_{\tau+3}$ with coordinates $k_1,\ldots,k_\tau$ removed. (That is, the coordinates of vectors in $V$ correspond to those entries that are non-$*$ in $\B$.) Since the $3$-slices of $T$ with indices $1,\ldots,\tau$ are linearly independent and have zeros outside the positions $k_1,\ldots,k_\tau$, we get $\dim V\leqslant 3$. We say that a $3$-slice is \textit{non-trivial} if it has a non-zero element somewhere except $k_1,\ldots,k_\tau$.

Therefore, if there were at least four $S_i$'s whose zeroth $3$-slices are non-trivial, they would be linearly dependent modulo the $3$-slices with positive indices. Using Lemma~\ref{lem75}, we would get that there are two $S_i$'s whose zeroth $3$-slices are non-zero and coincide up to scaling by non-zero elements of $\F$. The sum of these two $S_i$'s would still be a simple tensor with respect to $\F$, which would imply $\operatorname{rank}_{\F} T\leqslant \tau+2$ and contradict Lemma~\ref{lem76}. Therefore, there are at most three $S_i$'s with non-trivial zeroth $3$-slices, and the sum of these $3$-slices is a desired completion of $\B$.
\end{proof}

Lemmas~\ref{lem77} and~\ref{lem78} prove that $\operatorname{rank}_{\S} T(\B)\leqslant \tau+3$ if and only if $\B$ admits a completion of rank three with respect to an integral domain $\S$. By Corollary~\ref{cor74}, this happens if and only if the equations $f_1=0,\ldots,f_t=0$ have a simultaneous solution over $\S$. Since the tensor $T$ was constructed in polynomial time from a given family of polynomials, and since the entries of $T$ do not depend on the extension $\S$, the proof of Theorem~\ref{thrm2} is complete.

\section{Symmetric tensors}

The goal of this section is to prove Theorem~\ref{prob1} and, more generally, the analogue of Theorem~\ref{thrm2} for symmetric rank of matrices over a field. Our argument employs the reduction of the standard tensor rank problem to the symmetric version, and our proof uses Lemma~\ref{lemcompl} in a substantial way. Before we present our reduction, we need to specify some notation.

Throughout this section, we assume that $\F$ is a field.
We say that a tensor $T_0$ is obtained from $T\in\F^{I\times J\times K}$ by \textit{adjoining} a $3$-slice $A\in\F^{I\times J}$ if the $3$-slices of $T$ are precisely those of $T$ and $A$. We use similar definitions for adjoining $1$-slices and $2$-slices. For all $p,q\in I\cap J$, we define the $(p,q)$\textit{-unit} as the matrix $M\in\F^{I\times J}$ such that $M(i|j)=1$ if $i,j\in\{p,q\}$ and $M(i|j)=0$ otherwise. In particular, such a matrix becomes a conventional matrix unit whenever $p=q$. Now we are ready to present the main tool of this section.

\begin{defn}\label{defgadg2}
Let $I=\{i_1,\ldots,i_n\}$, $J=\{j_1,\ldots,j_n\}$, $K=\{k_1,\ldots,k_n\}$ be disjoint indexing sets, and let
$T\in\mathcal{F}^{I\times J\times K}$ be a tensor over a field $\F$. We define the tensor $S=S(T)\in\mathcal{F}^{H\times H\times H}$, where $H=I\cup J\cup K$, as follows:

\noindent (S1) $S(\alpha|\beta|\gamma)=T(i|j|k)$ if $(\alpha,\beta,\gamma)$ is a permutation of $(i,j,k)$ from $I\times J\times K$,

\noindent (S2) $S(\alpha_a|\beta_b|\gamma_c)=0$ otherwise.
\end{defn}

\begin{defn}\label{defredsym}
Let $I,J,K,H$ be the indexing sets as in Definition~\ref{defgadg2}, and let $S$ be a tensor in $\mathcal{F}^{H\times H\times H}$. We define $I^2$ as the set of all pairs $(i_p,i_q)$ with $1\leqslant p\leqslant q\leqslant n$. The sets $J^2, K^2$ are defined similarly, and we denote $\mathcal{H}=H\cup I^2\cup J^2\cup K^2$. We define the tensor $\mathcal{T}=\mathcal{T}(S)\in\mathcal{F}^{\mathcal{H}\times \mathcal{H}\times \mathcal{H}}$ by adjoining the $\pi$-unit $1$-slices, $\pi$-unit $2$-slices, and $\pi$-unit $3$-slices to $S$. Here, an index $\pi$ runs over the set $I^2\cup J^2\cup K^2$, and these $\pi$-unit slices get the index $\pi\in\H$.
\end{defn}

In the rest of the section, we are going to prove that
\begin{equation}\label{eqmainsym}
\operatorname{srank}_\F\,\mathcal{T}(S(T))=\operatorname{rank}_\F\,T+4.5(n^2+n)
\end{equation}
under a mild assumption on the cardinality of $\mathcal{F}$. This would show that $T\to\mathcal{T}(S(T))$ is a polynomial time many-one reduction from the standard rank problem to symmetric one. (We note that the assumption $|I|=|J|=|K|=n$ does not cause a loss of generality because a tensor that satisfies this assumption can be obtained from any tensor by adjoining zero slices, and this transformation does not change the rank.) In particular, this would give the analogue of Theorem~\ref{thrm2} for symmetric rank and prove Theorem~\ref{prob1}.

\begin{lem}\label{lemrank1}
Let $S(T)$ be the tensor as in Definition~\ref{defgadg2}, and $\mathcal{T}(S(T))$ be the tensor as in Definition~\ref{defredsym}.
Then $\operatorname{rank}_\F\,\mathcal{T}(S(T))\geqslant\operatorname{rank}_\F\,T+4.5(n^2+n)$.
\end{lem}

\begin{proof}
Let $M_3$ be a linear combination of the $3$-slices of $\mathcal{T}$ with indices in $I^2\cup J^2\cup K^2$. By Definition~\ref{defredsym}, all non-zero entries of these slices belong to the blocks $(I|I)$, $(J|J)$, $(K|K)$, and the same conclusion holds for $M_3$. Therefore, the addition of $M_3$ to any slice of $\mathcal{T}$ does not change its $(I|J|K)$ block. Similarly, the addition of any linear combination of the $1$-slices (or $2$-slices) with indices in $I^2\cup J^2\cup K^2$ to any $1$-slice (or $2$-slice, respectively) of $\mathcal{T}$ does not change the $(I|J|K)$ block.

We apply Lemma~\ref{lemcompl} to the $1$-slices with indices in $I^2\cup J^2\cup K^2$, then to the $2$-slices with these indices, and then to the $3$-slices. We get the desired inequality because the $(I|J|K)$ block of $\mathcal{T}$ is $T$ and because $3(|I^2|+|J^2|+|K^2|)=4.5(n^2+n)$.
\end{proof}

In particular, we see that $\operatorname{srank}_\F\,\mathcal{T}(S(T))\geqslant\operatorname{rank}_\F\,T+4.5(n^2+n)$. Our proof of the opposite inequality is more technical, and we need some more notation.

Let $T\in\F^{I\times I\times I}$ be a symmetric tensor, and let $\rho$ be a permutation of $I$, and let $(f_i)$ be a family of non-zero elements of $\F$. We say that the tensor whose $(i|j|k)$ entry equals $f_if_jf_kT(\rho(i)|\rho(j)|\rho(k))$ is obtained from $T$ by a \textit{monomial transformation}. The \textit{scaling} of $T$ is the multiplication of every entry of $T$ by a non-zero scalar in $\F$. Indices $i,\hat{\imath}\in I$ are called \textit{twins} with respect to $T$ if the $\hat{\imath}$th $1$-, $2$-, and $3$-slices are equal to the corresponding $i$th slices. \textit{Removing a twin} $\hat{\imath}$ is an operation that removes the $\hat{\imath}$th $1$-, $2$-, and $3$-slices from $T$.  

\begin{obs}\label{obstran}
Monomial transformations, scaling, and removing a twin leave the symmetric ranks of a given tensor invariant.
\end{obs}

\begin{ex}\label{ex2x2x2}
Let $|\F|\geqslant 9$ and $a\in\F$. Define the $2\times 2\times 2$ tensor $A$ by $A(1|1|1)=a$, $A(1|1|2)=A(1|2|1)=A(2|1|1)=1$, $A(1|2|2)=A(2|1|2)=A(2|2|1)=A(2|2|2)=0$. Then $\operatorname{srank}_\F\,A\leqslant 3$.
\end{ex}

\begin{proof}
One can check that the equation
\begin{equation}\label{eqdecsym}
A=\sum_{t=1}^3 s_t (1, r_t)\otimes (1, r_t)\otimes (1, r_t)
\end{equation}
holds over the rational function field $\F(q)$ if we set 
$$s_1=\frac{(-1 - q + q a)^3}{q (2 + 5 q + 2 q^2 - 3 q a - 3 q^2 a + q^2 a^2)},\,\,\, s_2 = \frac{1}{(-q +q^2) (-2 - q + q a)},$$
$$s_3 = \frac{q^2}{(q-1) (-1 - 2 q + q a)},\,\,\,
r_1 = \frac{q}{-1 - q + q a},\,\,\, r_2 = q,\,\,\, r_3 = 1.$$
If $a\neq 1$, the polynomials in the denominators do not vanish and have a common multiple of degree $7$. The assumption on the cardinality of $\F$ guarantees that, for some assignment $q\in\F$, the equality~\eqref{eqdecsym} is a valid decomposition. The result for $a=1$ follows now from Observation~\ref{obstran}.
\end{proof}

Let us define an $(i,j)$th $3$\textit{-transversal} of a tensor $T$ as the set of entries in which the first two coordinates are equal to $i$ and $j$, respectively. We define the notions of $1$- and $2$-transversals in a similar way. Now we are ready to prove the main technical lemma of this section.

\begin{lem}\label{lemmainsym}
Let $I,J,K,H$ be the indexing sets as in Definition~\ref{defgadg2}, and let $U$ be a symmetric tensor in $\mathcal{F}^{H\times H\times H}$. Assume that $U(i|j|k)=0$ whenever $i\in I$, $j\in J$, $k\in K$, and define the tensor $\mathcal{T}(U)$ as in Definition~\ref{defredsym}. If $|\F|\geqslant 9$, then $\operatorname{srank}_\F\,\mathcal{T}(U)\leqslant 4.5(n^2+n)$.
\end{lem}

\begin{proof}
For all pairs $\pi=(\alpha_p,\alpha_q)$, where $\alpha\in\{i,j,k\}$ and $1\leqslant p<q\leqslant n$, we define the tensor $\mathcal{L}_\pi\in\mathcal{F}^{\H\times \H\times \H}$ as follows. For arbitrary $r,s\in\{p,q\}$, we set


\noindent (L1) $\mathcal{L}_\pi(\alpha_r|\alpha_s|\beta_t)=U(\alpha_p|\alpha_q|\beta_t)$ if either $\beta\in\{i,j,k\}\setminus\{\alpha\}$ or $t>q$,

\noindent (L2) $\mathcal{L}_\pi(\alpha_r|\alpha_s|\pi)=1$,

\noindent (L3) $\mathcal{L}_\pi(z|x|y)=\mathcal{L}_\pi(y|z|x)=\mathcal{L}_\pi(x|y|z)$ if the latter is already defined,

\noindent (L4) the entries which are not yet defined are zero.

Every $\mathcal{L}_\pi$ can be reduced to the tensor as in Example~\ref{ex2x2x2} by transformations as in Observation~\ref{obstran}. We have $\operatorname{srank}_\F\,\mathcal{L}_\pi\leqslant 3$, so the result of the lemma would follow if we check that the tensor
$$\Phi=\mathcal{T}(U)-\sum\limits_{\alpha\in\{i,j,k\}}\sum\limits_{1\leqslant p<q\leqslant n}\mathcal{L}_{(\alpha_p,\alpha_q)}$$
has symmetric rank at most $9n$ with respect to $\F$. We can check that all the non-zero entries of $\Phi$ are covered by the $(u,u)$th $1$-, $2$-, and $3$-transversals, where $u$ runs over $\H$. We can write $\Phi=\sum_{u\in H}\mathcal{M}_u$, where $\mathcal{M}_u$ is defined as $\mathcal{M}_u(x|y|z)=\Phi(x|y|z)$ if $u$ appears at least twice among $x,y,z\in\H$ and $\mathcal{M}_u(x|y|z)=0$ otherwise. Finally, we get $\operatorname{srank}_\F\,\Phi\leqslant 9n$ because each of the $3n$ tensors $\mathcal{M}_u$ has symmetric rank at most three again by Observation~\ref{obstran} and Example~\ref{ex2x2x2}.
\end{proof}

\begin{lem}\label{lemrank2}
Let $S(T)$ be the tensor as in Definition~\ref{defgadg2}, and $\mathcal{T}(S(T))$ be the tensor as in Definition~\ref{defredsym}. If $|\F|\geqslant9$, then $\operatorname{srank}_\F\,\mathcal{T}(S(T))\leqslant\operatorname{rank}_\F\,T+4.5(n^2+n)$.
\end{lem}

\begin{proof}
Consider a decomposition $T=\sum_{t=1}^r a_t\otimes b_t\otimes c_t$, where $(a_t)$, $(b_t)$, $(c_t)$ are vectors in $\F^I$, $\F^J$, $\F^K$, respectively. We construct the family $(w_t)$ of vectors in $\F^\H$ by setting the $I$-part of $w_t$ equal to $a_t$, the $J$-part equal to $b_t$, the $K$-part equal to $c_t$, and setting all the other entries equal to zero. Now we see that the tensor
$$\mathcal{T}(S(T))-\sum_{t=1}^r w_t\otimes w_t\otimes w_t$$
satisfies the assumptions imposed on the tensor $\mathcal{T}(U)$ as in Lemma~\ref{lemmainsym}, and the application of this lemma completes the proof.
\end{proof}

Lemmas~\ref{lemrank1} and~\ref{lemrank2} prove that the equality~\eqref{eqmainsym} holds over any field of cardinality at least nine. Therefore, if the ring $\S$ as in Theorems~\ref{thrm2} and~\ref{cor80} is such a field, then the assertions of these theorems hold for the symmetric rank as well. In particular, we get Theorem~\ref{prob1}, and the proofs of all the results mentioned in Section~2 are now complete.


\begin{thebibliography}{99}


\bibitem{BGI}
A. Bernardi, A. Gimigliano, M. Ida, Computing symmetric rank for symmetric tensors, \textit{Journal of Symbolic Computation} 46 (2011) 34--53.

\bibitem{BK}
A. Bhangale, S. Kopparty, The complexity of computing the minimum rank of a sign pattern matrix, preprint (2015) arXiv:1503.04486.


\bibitem{Blas}
M. Bl\"{a}ser, Explicit tensors. In Perspectives in Computational Complexity, \textit{Progress in Computer Science and Applied Logic} 26 (2014) 117-130.

\bibitem{BCMT}
J. Brachat, P. Comon, B. Mourrain, E. Tsigaridas, Symmetric tensor decomposition. In \textit{17th European Signal Processing Conference}. IEEE, 2009. 525--529.

\bibitem{BFS}
J. F. Buss, G. S. Frandsen, J. O. Shallit, The computational complexity of some problems of linear algebra, \textit{Journal of Computer and System Sciences} 58 (1999) 572--596.

\bibitem{CP}
E. J. Candes, Y. Plan, Matrix completion with noise, \textit{Proceedings of the IEEE} 98 (2010) 925--936.

\bibitem{CR}
E. J. Candes, B. Recht, Exact matrix completion via convex optimization, \textit{Foundations of Computational mathematics} 9 (2009) 717--772.

\bibitem{CC}
J. D. Carroll, J. Chang, Analysis of individual differences in multidimensional scaling via an n-way generalization of Eckart–Young decomposition, \textit{Psychometrika} 35 (1970) 283--319.

\bibitem{CGLM}
P. Comon, G. Golub, L. H. Lim, B. Mourrain, Symmetric tensors and symmetric tensor rank, \textit{SIAM Journal on Matrix Analysis and Applications} 30 (2008) 1254--1279.

\bibitem{CLD}
P. Comon, X. Luciani, A. L. F. De Almeida,  Tensor decompositions, alternating least squares and other tales, \textit{Journal of Chemometrics} 23 (2009) 393--405.


\bibitem{dSLi}
V. de Silva, L. H. Lim, Tensor rank and the ill-posedness of the best low-rank approximation problem, \textit{SIAM Journal on Matrix Analysis and Applications} 30 (2008) 1084--1127.

\bibitem{Derk}
H. Derksen, Matrix completion and tensor rank, \textit{Linear \& Multilinear Algebra}, 64 (2016) 680--685.


\bibitem{GKPT}
I. Garcia-Marco, P. Koiran, T. Pecatte, S. Thomass\'{e}, On the complexity of partial derivatives, preprint (2016) arXiv:1607.05494.

\bibitem{GJ}
T. Gonzalez, J. Ja'Ja', On the complexity of computing bilinear forms with $\{0, 1\}$ constants, \textit{Journal of Computer and System Sciences} 20 (1980) 77--95.


\bibitem{Hast}
J. H\r{a}stad, Tensor rank is NP-complete, Journal of Algorithms 11 (1990) 644--654.

\bibitem{HL}
C. J. Hillar, L. H. Lim, Most tensor problems are NP-hard, \textit{Journal of the ACM} 60 (2013) 45.

\bibitem{Hitch}
F. L. Hitchcock, The expression of a tensor or a polyadic as a sum of products, \textit{Journal of Mathematics and Physics} 6 (1927) 164--189.

\bibitem{HK}
J. E. Hopcroft, L. R. Kerr, On minimizing the number of multiplications necessary for matrix multiplication, \textit{SIAM Journal on Applied Mathematics} 20 (1971) 30--36.

\bibitem{Karp}
R. Karp, Reducibility Among Combinatorial Problems, Proceedings of the Symposium on the Complexity of Computer Computations (1972) 85--103.

\bibitem{onemore}
R. H. Keshavan, S. Oh, A. Montanari, Matrix completion from a few entries, in \textit{Proceedings of the 2009 IEEE International Symposium on Information Theory}, IEEE, 2009. Pages 324--328.

\bibitem{KR}
K. H. Kim, F. W. Roush, Problems equivalent to rational Diophantine solvability, \textit{Journal of Algebra} 124 (1989) 493--505.

\bibitem{Koenin}
J. Koenigsmann,  Defining ${\mathbb Z}$ in ${\mathbb Q}$, \textit{Annals of Mathematics} 183 (2016) 73--93.

\bibitem{KB}
T. G. Kolda, B. W. Bader,  Tensor decompositions and applications, \textit{SIAM Review} 51 (2009) 455--500.

\bibitem{Lan}
J. M. Landsberg. \textit{Tensors: geometry and applications}. American Mathematical Society, Providence, RI, USA, 2012.

\bibitem{Laur}
M. Laurent, Matrix Completion Problems, \textit{Encyclopedia of Optimization} 3 (2009) 221--229.

\bibitem{LC}
L. H. Lim, P. Comon, Multiarray signal processing: Tensor decomposition meets compressed sensing, \textit{Comptes Rendus Mecanique} 338 (2010) 311--320.

\bibitem{MR}
Y. Matiyasevich, J. Robinson, Reduction of an arbitrary Diophantine equation to one in $13$ unknowns, \textit{Acta Arithmetica} 27 (1975) 521--549.

\bibitem{Mato}
J. Matou\v{s}ek, Intersection graphs of segments and $\exists\mathbb{R}$, preprint (2014) arXiv:1406.2326.

\bibitem{Maz}
B. Mazur, Questions of decidability and undecidability in number theory, \textit{Journal of Symbolic Logic} 59 (1994) 353--371.


\bibitem{OO}
L. Oeding, G. Ottaviani, Eigenvectors of tensors and algorithms for Waring decomposition, \textit{Journal of Symbolic Computation} 54 (2013) 9--35.

\bibitem{Peet}
R. Peeters, Orthogonal representations over finite fields and the chromatic number of graphs, \textit{Combinatorica} 16 (1996) 417--431.

\bibitem{Poon}
B. Poonen, Hilbert’s tenth problem and Mazur’s conjecture for large subrings of $\mathbb{Q}$, \textit{Journal of the American Mathematical Society} 16 (2003) 981--990.


\bibitem{Robeva}
E. Robeva, Orthogonal decomposition of symmetric tensors, \textit{SIAM Journal on Matrix Analysis and Applications} 37 (2016) 86--102.

\bibitem{SH}
A. Shashua, T. Hazan, Non-negative tensor factorization with applications to statistics and computer vision. In \textit{Proceedings of the 22nd international conference on Machine learning}. ACM, 2005. 792--799.

\bibitem{myfpm}
Y. Shitov, On the coincidence of the factor and Gondran–Minoux rank functions of matrices over a semiring, \textit{Journal of Mathematical Sciences} 193 (2013) 802--808.

\bibitem{mypsd}
Y. Shitov, The complexity of positive semidefinite matrix factorization, preprint (2016) arXiv:1606.09065.

\bibitem{SBG}
N. D. Sidiropoulos, R. Bro, G. B. Giannakis, Parallel factor analysis in sensor array processing, \textit{IEEE Transactions on Signal Processing} 48 (2000) 2377--2388.

\bibitem{Smo}
P. Smolensky, Harmony in linguistic cognition, \textit{Cognitive Science} 30 (2006) 779--801.

\bibitem{Str}
V. Strassen, The asymptotic spectrum of tensors, \textit{J. Reine Angew. Math.} 384 (1988) 102--152.
\end{thebibliography}
\end{document}